\documentclass{article}

\usepackage{
	amsmath, 
	amsthm, 
	amssymb,
	fancyhdr, 
	setspace,
	tikz-cd,
}
\usepackage{upgreek}

\usepackage{lscape}
\usepackage{color}
\usepackage[color,matrix,arrow]{xy}
\usepackage{setspace}
\xyoption{all}

\DeclareMathOperator{\Ext}{Ext}
\DeclareMathOperator{\Hom}{Hom}
\DeclareMathOperator{\pd}{pd}
\DeclareMathOperator{\id}{id}
\DeclareMathOperator{\add}{add}
\DeclareMathOperator{\Gen}{Gen}
\DeclareMathOperator{\Cogen}{Cogen}

\usepackage{avant}
\usepackage{amsmath, amssymb, amsfonts, color, tikz, bbm, txfonts, euscript}
\usepackage{amssymb}
\usepackage{enumerate}
\usepackage{bm}
\newtheorem{theorem}{Theorem}[section]
\newtheorem{lemma}[theorem]{Lemma}

\newtheorem{prop}[theorem]{Proposition}
\newtheorem{cor}[theorem]{Corollary}
\theoremstyle{definition}
\newtheorem{mydef}[theorem]{Definition}

\begin{document}

\thispagestyle{empty}

\title{Three Results Concerning Auslander Algebras}
\author{Stephen Zito}
        
\maketitle

\begin{abstract}
Our first result provides a new characterization of Auslander algebras using a property of hereditary torsion pairs.  The second result shows an Auslander algebra $\Lambda$ is left or right glued if and only if $\Lambda$ is representation-finite.  Finally, our third result shows the module category of any Auslander algebra contains 
a tilting module with a particular property, which we call the hereditary property.  Applications of this property are investigated.

\end{abstract}

\section{Introduction}
We set the notation for the remainder of this paper. All algebras are assumed to be finite dimensional over an algebraically closed field $k$.  If $\Lambda$ is a $k$-algebra then denote by $\mathop{\text{mod}}\Lambda$ the category of finitely generated right $\Lambda$-modules and by $\mathop{\text{ind}}\Lambda$ a set of representatives of each isomorphism class of indecomposable right $\Lambda$-modules.  Given $M\in\mathop{\text{mod}}\Lambda$, the projective dimension of $M$ in $\mathop{\text{mod}}\Lambda$ is denoted by $\pd_{\Lambda}M$ and its injective dimension by $\id_{\Lambda}M$.  We denote by $\add M$ the smallest additive full subcategory of $\mathop{\text{mod}}\Lambda$ containing $M$, that is, the full subcategory of $\mathop{\text{mod}}\Lambda$ whose objects are the direct sums of direct summands of the module $M$.  We let $\tau_{\Lambda}$ and $\tau^{-1}_{\Lambda}$ be the Auslander-Reiten translations in $\mathop{\text{mod}}\Lambda$.  $D$ will denote the standard duality functor $\Hom_k(-,k)$.  Finally,  let $\mathop{\text{gl.dim}}\Lambda$ stand for the global dimension and $\mathop{\text{domdim}}\Lambda$ stand for the dominant dimension of an algebra $\Lambda$ (see Definition $\ref{def3}$).

Let $\Lambda$ be an algebra of finite type and $M_1,M_2,\cdots,M_n$ be a complete set of representatives of the isomorphism classes of indecomposable $\Lambda$-modules.  Then  $A=\text{End}_{\Lambda}(\oplus_{i=1}^nM_i)$ is the Auslander algebra of $\Lambda$ (see Definition $\ref{Auslander}$). 
\par
 Let $\mathcal{C}_{\Lambda}$ be the full subcategory of $\mathop{\text{mod}}\Lambda$ consisting of all modules generated and cogenerated by the direct sum of representatives of the isomorphism classes of all indecomposable projective-injective $\Lambda$-modules (see Definition $\ref{Q}$).  When $\mathop{\text{gl.dim}}\Lambda=2$, Crawley-Boevey and Sauter showed in $\cite{CBS}$ that the algebra $\Lambda$ is an Auslander algebra if and only if there exists a tilting $\Lambda$-module $T_{\mathcal{C}}$ in $\mathcal{C}_{\Lambda}$. 
\par
Work by Nguyen, Reiten, Todorov, and Zhu in $\cite{NRTZ}$ showed the existence of such a tilting module is equivalent to the dominant dimension being at least $2$ without any condition on the global dimension of $\Lambda$.  They also gave a precise description of such a tilting module.
\par
Our first result provides a new characterization of Auslander algebras  using a property of hereditary torsion pairs (see Definition $\ref{Her}$).  In what follows, let $\mathcal{P}^1(\Lambda)=\{M\in\mathop{\text{mod}}\Lambda~|~\pd_{\Lambda}M\leq1\}$.
\begin{theorem}
Let $\Lambda$ be an algebra.  Then $\Lambda$ is an Auslander algebra if and only if there exists a hereditary torsion pair $(\mathcal{T},\mathcal{F})$ such that $\mathcal{P}^1(\Lambda)=\mathcal{F}$ and, if $I$ is any indecomposable injective $\Lambda$-module, then $\pd_{\Lambda}I\neq1$.
\end{theorem}
Our second result concerns Auslander algebras which are also left or right glued algebras.  These algebras were introduced by Assem and Coelho in $\cite{AC}$ where a convenient homological characterization was proved (see Theorem $\ref{left}$).  Trivially, any representation-finite algebra is left and right glued and we prove this is the only case for Auslander algebras.
\begin{theorem}
Let $\Lambda$ be an Auslander algebra.  Then $\Lambda$ is left or right glued if and only if $\Lambda$ is representation-finite.
\end{theorem}
Our third result concerns tilting modules possessing a certain property.  We introduce the hereditary property (see Definition $\ref{MINE}$).  We show that the module category of any Auslander Algebra contains a tilting module with the hereditary property.
\begin{theorem}
Let $\Lambda$ be an Auslander algebra.  Then there exists a tilting $\Lambda$-module with the hereditary property.
\end{theorem}
We further investigate this property.  In particular, we prove a characterization of hereditary algebras.
\begin{theorem}
Let $\Lambda$ be an algebra.  Then $\Lambda$ is hereditary if and only if every tilting $\Lambda$-module possesses the hereditary property.
\end{theorem}

\subsection{Tilting and Cotilting Modules}
  We begin with the definition of tilting and cotilting modules.
   \begin{mydef}
   \label{Tilting}
    Let $\Lambda$ be an algebra.  A $\Lambda$-module $T$ is a $\emph{partial tilting module}$ if the following two conditions are satisfied: 
   \begin{enumerate}
   \item[($\text{1}$)] $\pd_{\Lambda}T\leq1$.
   \item[($\text{2}$)] $\Ext_{\Lambda}^1(T,T)=0$.
   \end{enumerate}
   A partial tilting module $T$ is called a $\emph{tilting module}$ if it also satisfies:
   \begin{enumerate}
   \item[($\text{3}$)] There exists a short exact sequence $0\rightarrow \Lambda_{\Lambda}\rightarrow T'\rightarrow T''\rightarrow 0$ in $\mathop{\text{mod}}\Lambda$ with $T'$ and $T''$ $\in \add T$.
   \end{enumerate}
   A $\Lambda$-module $C$ is a $\emph{partial cotilting module}$ if the following two conditions are satisfied: 
   \begin{enumerate}
   \item[($\text{1}'$)] $\id_{\Lambda}C\leq1$.
   \item[($\text{2}'$)] $\Ext_{\Lambda}^1(C,C)=0$.
   \end{enumerate}   
   A partial cotilting module is called a $\emph{cotilting module}$ if it also satisfies:
   \begin{enumerate}
   \item[($\text{3}'$)] There exists a short exact sequence $0\rightarrow C'\rightarrow C''\rightarrow D\Lambda_{\Lambda}\rightarrow 0$ in $\mathop{\text{mod}}\Lambda$ with $C'$ and $C''$ $\in \add C$.   
  \end{enumerate} 
   \end{mydef}
 
 Tilting modules and cotilting modules induce torsion pairs in a natural way.  
 \begin{mydef} 
 \label{def5}
 A pair of full subcategories $(\mathcal{T},\mathcal{F})$ of $\mathop{\text{mod}}\Lambda$ is called a $\emph{torsion pair}$ if the following conditions are satisfied:
   \begin{enumerate}
   \item[(1)] $\text{Hom}_{\Lambda}(M,N)=0$ for all $M\in\mathcal{T}$, $N\in\mathcal{F}.$
   \item[(2)] $\text{Hom}_{\Lambda}(M,-)|_\mathcal{F}=0$ implies $M\in\mathcal{T}.$
   \item[(3)] $\text{Hom}_{\Lambda}(-,N)|_\mathcal{T}=0$ implies $N\in\mathcal{F}.$
   \end{enumerate}
   \end{mydef}
 We say $\mathcal{T}$ is a $\it{torsion~class}$ while $\mathcal{F}$ is a $\it{torsion-free~class}$.  It can be shown that $\mathcal{T}$ is closed under images, direct sums, and extensions while $\mathcal{F}$ is closed under submodules, direct products, and extensions.  See $\cite{ASS}$ for more details.
 \begin{mydef}
\label{Her}
We say a torsion pair $(\mathcal{T},\mathcal{F})$ is $\emph{hereditary}$ if $\mathcal{T}$ is closed under submodules.  This is equivalent to $\mathcal{F}$ being closed under injective envelopes.
\end{mydef} 
 We say a torsion pair $(\mathcal{T},\mathcal{F})$
 is $\it{splitting}$ if every indecomposable $\Lambda$-module belongs to either $\mathcal{T}$ or $\mathcal{F}$.  We have the following characterization.
  \begin{prop}$\emph{\cite[VI,~Proposition~1.7]{ASS}}$
 \label{split}
 Let $(\mathcal{T},\mathcal{F})$ be a torsion pair in $\mathop{\emph{mod}}\Lambda$.  The following are equivalent:
 \begin{enumerate}
 \item[$\emph{(a)}$] $(\mathcal{T},\mathcal{F})$ is splitting.
 \item[$\emph{(b)}$] If $M\in\mathcal{T}$, then $\tau_{\Lambda}^{-1}M\in\mathcal{T}$.
 \item[$\emph{(c)}$] If $N\in\mathcal{F}$, then $\tau_{\Lambda}N\in\mathcal{F}$.
 \end{enumerate}
 \end{prop}

 \begin{mydef}
\label{Gen/Cogen}
Let $M$ be a $\Lambda$-module.  We define $\mathop{Gen} M$ to be the class of all modules $X$ in $\mathop{\text{mod}}\Lambda$ generated by $M$, that is, the modules $X$ such that there exists an integer $d\geq0$ and an epimorphism $M^d\rightarrow X$ of $\Lambda$-modules.  Here, $M^d$ is the direct sum of $d$ copies of $M$.  Dually, we define $\mathop{Cogen}M$ to be the class of all modules $Y$ in $\mathop{\text{mod}}\Lambda$ cogenerated  by $M$, that is, the modules $Y$ such that there exist an integer $d\geq0$ and a monomorphism $Y\rightarrow M^d$ of $\Lambda$-modules.
\end{mydef}

Now, consider the following full subcategories of $\mathop{\text{mod}}\Lambda$ where $T$ is a tilting $\Lambda$-module.
 \[
 \mathcal{T}(T)=\{M\in\mathop{\text{mod}}\Lambda~|~ \text{Ext}_{\Lambda}^{1}(T,M)=0\}
 \]
 \[
 \mathcal{F}(T)=\{M\in\mathop{\text{mod}}\Lambda~|~\text{Hom}_{\Lambda}(T,M)=0\}
 \]
 Then $(\text{Gen}T,\mathcal{F}(T))=(\mathcal{T}(T),\text{Cogen}(\tau_{\Lambda} T))$ is a torsion pair in $\mathop{\text{mod}}\Lambda$. 
 Consider the following full subcategories of $\mathop{\text{mod}}\Lambda$ where $C$ is a cotilting $\Lambda$-module. 
  \[
 \mathcal{F}(C)=\{M\in\mathop{\text{mod}}\Lambda~|~ \text{Ext}_{\Lambda}^{1}(M,C)=0\}
 \]
 \[
 \mathcal{T}(C)=\{M\in\mathop{\text{mod}}\Lambda~|~\text{Hom}_{\Lambda}(M,C)=0\}
 \] 
 Then $(\text{Gen}(\tau_{\Lambda}^{-1}C),\mathcal{F}(C))=(\mathcal{T}(C),\text{Cogen}(C))$ is a torsion pair in $\mathop{\text{mod}}\Lambda$.  Once again, we refer the reader to $\cite{ASS}$ for more details.
We need two properties of tilting modules.  We start with a definition.

\begin{mydef}
Let $\mathcal{T}$ be a full subcategory of $\mathop{\text{mod}}\Lambda$.  We say a $\Lambda$-module $X\in\mathcal{T}$ is $\text{Ext}$-$\it{projective}$ if $\text{Ext}_{\Lambda}^1(X,-)|_{\mathcal{T}}=0$.
\end{mydef} 

 \begin{prop}$\emph{\cite[VI.2,~Theorem~2.5]{ASS}}$ 
 \label{Tilting}
 Let $\Lambda$ be an algebra and $T$ a tilting $\Lambda$-module with $(\mathcal{T}(T),\mathcal{F}(T))$ the induced torsion pair.  Then, for every $X\in\mathop{\emph{mod}}\Lambda$, $X\in\add T$ if and only if $X$ is $\emph{Ext}$-projective in $\mathcal{T}(T)$.
 \end{prop}

\begin{prop}$\emph{\cite[VI.2,~Theorem~2.5]{ASS}}$
\label{need}
Let $\Lambda$ be an algebra and $T$ a tilting $\Lambda$-module with $(\mathcal{T}(T),\mathcal{F}(T))$ the induced torsion pair.  Then, for every module $M\in\mathcal{T}(T)$, there exists a short exact sequence 
$0\rightarrow L \rightarrow T_0\rightarrow M\rightarrow 0$ with $T_0\in\add T$ and $L\in\mathcal{T}(T)$.
\end{prop}

\subsection{Properties of the Subcategory $\mathcal{C}_{\Lambda}$}
Let $\Lambda$ be an algebra.
\begin{mydef}
\label{Q}
Let $\tilde{Q}$ be the direct sum of representatives of the isomorphism classes of all indecomposable projective-injective $\Lambda$-modules.  Let $\mathcal{C}_{\Lambda}:=(\text{Gen}\tilde{Q})\cap(\text{Cogen}\tilde{Q})$ be the full subcategory consisting of all modules generated and cogenerated by $\tilde{Q}$.
\end{mydef}
Nguyen, Reiten, Todorov, and Zhu in $\cite{NRTZ}$ studied the existence of a tilting module in $\mathcal{C}_{\Lambda}$.  We need three of their preliminary results.
\begin{prop}$\emph{\cite[Proposition 1.1.3]{NRTZ}}$
\label{prelim1}
If $P$ is projective and $P$ is in $\mathcal{C}_{\Lambda}$, then $P$ is projective-injective.  If $I$ is injective and $I$ is in $\mathcal{C}_{\Lambda}$, then $I$ is projective-injective.
\end{prop}
\begin{lemma}$\emph{\cite[Lemma 1.1.4]{NRTZ}}$
\label{prelim2}
Let $X$ be in $\mathcal{C}_{\Lambda}$.  Let $Y$ be a $\Lambda$-module with $\pd_{\Lambda}Y=1$.  Then $\Ext_{\Lambda}^1(Y,X)=0$.
\end{lemma}

\begin{prop}$\emph{\cite[Proposition 1.2.5]{NRTZ}}$
\label{prelim3}
Let $\Lambda$ be an algebra.  Let $\tilde{Q}$ and $\mathcal{C}_{\Lambda}$ be defined as above.
Let $\{X_i\}_{i\in I}$ be the set of representatives of indecomposable modules in $\mathcal{C}_{\Lambda}$ such that $\pd_{\Lambda}X_i=1$.  Then:
\begin{enumerate}
 \item[\emph{(1)}] The set $\{X_i\}_{i\in I}$ is finite, that is $I=\{1,2,\cdots,s\}$ for some $s<\infty$.
  \item[\emph{(2)}] Let $X=\bigoplus_{i=1}^sX_i$.  Then $\tilde{Q}\oplus X$ is a partial tilting module.
  \item[\emph{(3)}] If there is a tilting module $T_{\mathcal{C}}$ in $\mathcal{C}_{\Lambda}$, then $T_{\mathcal{C}}=\tilde{Q}\oplus X$.
 \item[\emph{(4)}] If there is a tilting module $T_{\mathcal{C}}$ in $\mathcal{C}_{\Lambda}$, then $T_{\mathcal{C}}$ is unique.
 \end{enumerate}
\end{prop}
Nguyen, Reiten, Todorov, and Zhu in $\cite{NRTZ}$ showed the existence of such a tilting module without any condition on the global dimension of $\Lambda$ and gave a precise description.

\subsection{Auslander Algebras}
We begin with the definition of Auslander algebras.
\begin{mydef}
\label{Auslander}
Let $\Lambda$ be an algebra of finite type and $M_1,M_2,\cdots,M_n$ be a complete set of representatives of the isomorphism classes of indecomposable $\Lambda$-modules.  Then  $A=\text{End}_{\Lambda}(\oplus_{i=1}^nM_i)$ is the $\emph{Auslander algebra}$ of $\Lambda$. 
\end{mydef}
Auslander in $\cite{AU}$ characterized the algebras which arise this way as algebras of global dimension at most $2$ and $\emph{dominant dimension}$ at least $2$.  We now recall the definition of dominant dimension.
\begin{mydef}
\label{def3}
Let $\Lambda$ be an algebra and let
\begin{center}
$0\rightarrow\Lambda_{\Lambda}\rightarrow I_0\rightarrow I_1\rightarrow\ I_2\rightarrow\cdots$
\end{center}
be a minimal injective resolution of $\Lambda$.  Then $\mathop{\text{domdim}}\Lambda=n$ if $I_i$ is projective for $0\leq i\leq n-1$ and $I_n$ is not projective.  If all $I_n$ are projective, we say $\mathop{\text{domdim}}\Lambda=\infty$.
\end{mydef}
When $\mathop{\text{gl.dim}}\Lambda=2$, Crawley-Boevey and Sauter showed the following characterization of Auslander algebras.
\begin{lemma}$\emph{\cite[Lemma~1.1]{CBS}}$
\label{CBS}
If $\mathop{\emph{gl.dim}}\Lambda=2$, then $\mathcal{C}_{\Lambda}$ contains a tilting-cotilting module if and only if $\Lambda$ is an Auslander algebra.
\end{lemma}
Another characterization was given by Li and Zhang in $\cite{LZ}$.  Recall, $\tilde{Q}$ is the direct sum of representatives of the isomorphism classes of all indecomposable projective-injective $\Lambda$-modules.
\begin{theorem}$\emph{\cite[Theorem~4.1]{LZ}}$
\label{char}
Let $\Lambda$ be a finite dimensional $k$-algebra.  Then $\Lambda$ is an Auslander algebra if and only if $\mathop{\emph{gl.dim}}\Lambda\leq2$ and $\add\tilde{Q}=\{I\in\mathop{\emph{mod}}\Lambda~|~\id_{\Lambda}I=0~and~\pd_{\Lambda}\mathop{\emph{soc}}I\leq1\}$.
\end{theorem}

We need the following property of Auslander algebras.  Let $\mathcal{P}^1(\Lambda)$ be defined as before while $\mathcal{I}^1(\Lambda)=\{M\in\mathop{\text{mod}}\Lambda~|~\id_{\Lambda}M\leq1\}$.
\begin{prop}
\label{mine}
Let ${\Lambda}$ be an Auslander algebra with $T_{\mathcal{C}}$ the tilting-cotilting module in $\mathcal{C}_{\Lambda}$.  Then $\mathcal{P}^1(\Lambda)=\emph{Cogen}(T_{\mathcal{C}})$ and $\mathcal{I}^1(\Lambda)=\emph{Gen}(T_{\mathcal{C}})$.
\end{prop}
\begin{proof}
The proof of $\mathcal{P}^1(\Lambda)=\text{Cogen}(T_{\mathcal{C}})$ is contained in Proposition 2.2 from $\cite{ZITO3}$.  Thus, we will show $\mathcal{I}^1(\Lambda)=\text{Gen}(T_{\mathcal{C}})$, which is similar.  Let $X\in\mathcal{I}^1(\Lambda)$.  If $X$ is injective, then $X\in\text{Gen}(T_{\mathcal{C}})$ by the definition of a tilting module.  If $\id_{\Lambda}X=1$, then the dual statement of Lemma $\ref{prelim2}$ gives $\Ext_{\Lambda}^1(T_{\mathcal{C}},X)=0$.  Since $T_{\mathcal{C}}$ is a tilting module, we must have $X\in\text{Gen}(T_{\mathcal{C}})$.
\par
Next, assume $X\in\text{Gen}(T_{\mathcal{C}})$ but $X\notin\mathcal{I}^1(\Lambda)$.  Since $\Lambda$ is an Auslander algebra, we must have $\id_{\Lambda}X=2$.  Since $X\in\text{Gen}(T_{\mathcal{C}})$, there exists a short exact sequence $0\rightarrow Y \rightarrow T_{\mathcal{C}}'\rightarrow X\rightarrow 0$ in $\mathop{\text{mod}}\Lambda$ with $T_{\mathcal{C}}' \in \add T_{\mathcal{C}}$.  Now, $\id_{\Lambda}T_{\mathcal{C}}'\leq1$ and $\id_{\Lambda}X=2$ imply $\id_{\Lambda}Y=3$, which is a contradiction to $\Lambda$ being an Auslander algebra.  Thus, $X\in\mathcal{I}^1(\Lambda)$.
\end{proof}

\subsection{Miscellaneous Results}
In this subsection, we gather all remaining theorems and propositions needed for our main results.  We begin with the notion of $left~(right)~glued~algebra$, introduced by Assem and Coelho in $\cite{AC}$.  This type of algebra is a finite enlargement in the postprojective (or preinjective) components of a finite set of tilted algebras having complete slices in these components.  For our purposes, we need the following homological characterization proved by Assem and Coelho.
\begin{theorem}$\emph{\cite[Theorem~3.2]{AC}}$
\label{left}
\begin{enumerate}
\item[\emph{(1)}] An algebra $\Lambda$ is left glued if and only if $\id_{\Lambda}M=1$ for almost all non-isomorphic indecomposable $\Lambda$-modules $M$  
\item[\emph{(2)}] An algebra $\Lambda$ is right glued if and only if $\pd_{\Lambda}M=1$ for almost all non-isomorphic indecomposable $\Lambda$-modules $M$.
\end{enumerate}
\end{theorem}
Next, we wish to compute the global dimension of an algebra.  The following theorem due to Auslander is very useful.
\begin{theorem}$\emph{\cite{AU1}}$
\label{Global}
If $\Lambda$ is an algebra, then 
\[
\mathop{\emph{gl.dim}}\Lambda=1+\emph{max}\{\pd_{\Lambda}(\emph{rad}e\Lambda); e\in\Lambda~is~a~primitive~idempotent\}.
\]
\[
=\emph{max}\{\pd_{\Lambda}S; S~is~a~simple~\Lambda-module\}
\]
\end{theorem}
A module is said to be $\it{torsionless}$ provided it can be embedded into a projective module.  A module is said to be $\it{co}$-$\it{torsionless}$ provided it is a factor module of an injective module.  An algebra $\Lambda$ is $\it{torsionless}$-$\it{finite}$ provided there are only finitely many isomorphism classes of indecomposable torsionless $\Lambda$-modules.  Given such an algebra $\Lambda$, we have the following combinatorial relationship between the torsionless and co-torsionless modules.
\begin{prop}$\emph{\cite[Corollary~5]{R}}$
\label{torsion}
If $\Lambda$ is torsionless-finite, the number of isomorphism classes of indecomposable factor modules of injective modules is equal to the number of isomorphism classes of indecomposable torsionless modules.
\end{prop}

\section{Main Results}
We begin with a new characterization of Auslander algebras.  Let $\Lambda$ be an algebra.  Recall, 
$\mathcal{P}^1(\Lambda)=\{M\in\mathop{\text{mod}}\Lambda~|~\pd_{\Lambda}M\leq1\}$.

\begin{theorem}
$\Lambda$ is an Auslander algebra if and only if there exists a hereditary torsion pair $(\mathcal{T},\mathcal{F})$ in $\mathop{\emph{mod}}\Lambda$ such that $\mathcal{P}^1(\Lambda)=\mathcal{F}$ and, if $I$ is any indecomposable injective $\Lambda$-module, then $\pd_{\Lambda}I\neq1$.
\end{theorem}
\begin{proof}
Assume $\Lambda$ is an Auslander algebra.  By Lemma $\ref{CBS}$, there exists a tilting-cotilting $\Lambda$-module $T_{\mathcal{C}}$ such that $T_{\mathcal{C}}\in\mathcal{C}_{\Lambda}$.  Since $T_{\mathcal{C}}$ is a cotilting $\Lambda$-module, it induces a torsion pair, $(\mathcal{T}(T_{\mathcal{C}}),\mathcal{F}(T_{\mathcal{C}}))$, such that $\mathcal{F}(T_{\mathcal{C}})=\text{Cogen}(T_{\mathcal{C}})$.  By the definition of $\mathcal{C}_{\Lambda}$, we see $\mathcal{F}(T_{\mathcal{C}})$ is closed under injective envelopes and Definition $\ref{Her}$ gives $(\mathcal{T}(T_{\mathcal{C}}),\mathcal{F}(T_{\mathcal{C}}))$ is a hereditary torsion pair.  Proposition $\ref{mine}$ gives us $\mathcal{P}^1(\Lambda)=\text{Cogen}(T_{\mathcal{C}})=\mathcal{F}(T)$.  Finally, let $I$ be an indecomposable injective $\Lambda$-module and suppose $\pd_{\Lambda}I=1$.  Then $I\in\mathcal{F}(T_{\mathcal{C}})$.  Since $\mathcal{F}(T_{\mathcal{C}})=\text{Cogen}(T_{\mathcal{C}})$ and $I$ is injective, we must have $I\in\add T_{\mathcal{C}}$.  By Proposition $\ref{prelim1}$, $I$ must be projective and this contradicts $\pd_{\Lambda}I=1$.  We conclude $\pd_{\Lambda}I\neq1$.
\par
Now, assume $\Lambda$ is an algebra such that there exists a hereditary torsion pair $(\mathcal{T},\mathcal{F})$ in $\mathop{\text{mod}}\Lambda$ with $\mathcal{P}^1(\Lambda)=\mathcal{F}$ and, if $I$ is any indecomposable injective $\Lambda$-module, then $\pd_{\Lambda}I\neq1$.  By assumption, every indecomposable projective module $P\in\mathcal{F}$.  Since $\mathcal{F}$ is a torsion-free class, this further implies $\mathop{\text{rad}}P\in\mathcal{F}$.  By Theorem $\ref{Global}$,  $\mathop{\text{gl.dim}}\Lambda\leq2$.  Let $I$ be an indecomposable injective module and consider $\mathop{\text{soc}}I$.  If $\pd_{\Lambda}\mathop{\text{soc}}I\leq1$, then $\mathop{\text{soc}}I\in\mathcal{F}$.  Since $\mathcal{F}$ is closed under injective envelopes, $I\in\mathcal{F}$.  Our assumption $\pd_{\Lambda}I\neq1$ implies $I$ is projective.  Using Theorem $\ref{char}$, we conclude $\Lambda$ is an Auslander algebra. 
\end{proof}
Our next result concerns any Auslander algebra $\Lambda$ which is left or right glued.  We show this is the case only when $\Lambda$ is representation-finte.
\begin{theorem}
Let $\Lambda$ be an Auslander algebra.  Then $\Lambda$ is left or right glued if and only if $\Lambda$ is representation-finite.
\end{theorem}
\begin{proof}
Obviously, if $\Lambda$ is representation-finite, then $\Lambda$ is left and right glued.  We may assume $\Lambda$ is left glued with the case $\Lambda$ being right glued similar.  Since $\Lambda$ is left glued, Theorem $\ref{left}$ gives $\id_{\Lambda}M=1$ for almost all non-isomorphic indecomposable $\Lambda$-modules $M$.  This implies the set of all indecomposable modules $M$ such that $\id_{\Lambda}M=2$ and $\pd_{\Lambda}M\leq2$ is finite.  Now consider the set of all indecomposable modules $M$ such that $\id_{\Lambda}M=1$ and $\pd_{\Lambda}M=1$.  Any such module $M$ must belong to $\mathcal{C}_{\Lambda}$ by Proposition $\ref{mine}$ and this set must be finite by Proposition $\ref{prelim3}$.  Thus, we have shown that $\mathcal{P}^1(\Lambda)=\{M\in\mathop{\text{mod}}\Lambda~|~\pd_{\Lambda}M\leq1\}$ is a finite set.  Since $\mathcal{P}^1(\Lambda)=\text{Cogen}(T_{\mathcal{C}})$, again by Proposition $\ref{mine}$, we have shown $\Lambda$ is torsionless-finite.  Combining Proposition $\ref{torsion}$ with yet another application of Proposition $\ref{mine}$ finally gives $\Lambda$ is representation-finite.
\end{proof}
Our final main result involves tilting modules possessing a certain property which we now introduce and define.
\begin{mydef}
\label{MINE}
Let $\Lambda$ be an algebra with $T$ a tilting $\Lambda$-module.  Then $T$ possesses the $\it{hereditary}$ $\it{property}$ if, for every module $M\in\Gen{T}$, there exists a short exact sequence
$0\rightarrow T'\rightarrow T''\rightarrow M\rightarrow 0$ in $\mathop{\text{mod}}\Lambda$ with $T'$ and $T''$ $\in \add T$

\end{mydef}
Dually, one can define the $\it{co}$-$\it{hereditary}$ $\it{property}$ for a cotilting module $C$ and any module $M\in\Cogen C$.  We show that the module category of an Auslander algebra $\Lambda$ contains a tilting module with the hereditary property.  
\begin{theorem}
Let $\Lambda$ be an Auslander algebra.  Then there exists a tilting $\Lambda$-module with the hereditary property.
\end{theorem}
\begin{proof}
Consider the tilting module $T_{\mathcal{C}}$ guaranteed by Lemma $\ref{CBS}$.  We will show $T_{\mathcal{C}}$ possesses the hereditary property.  Proposition $\ref{need}$ gives, for every $M\in\Gen T_{\mathcal{C}}$, a short exact sequence $0\rightarrow L \rightarrow T_{\mathcal{C}}'\rightarrow M\rightarrow 0$ with $T_{\mathcal{C}}'\in\add T$ and $L\in\Gen(T_{\mathcal{C}})$.  Clearly, $L\in\Cogen(T_{\mathcal{C}})$ and Proposition $\ref{mine}$ shows $\pd_{\Lambda}L=1$.  Thus, we have a module $L$ such that $\pd_{\Lambda}L=1$ and $L\in\mathcal{C}_{\Lambda}$.  Let $L'$ be any indecomposable summand of $L$.  Applying Proposition $\ref{prelim3}$, we see $L'\in\add T_{\mathcal{C}}$.  Since $L'$ was arbitrary, we conclude $L\in\add T_{\mathcal{C}}$ and $T_{\mathcal{C}}$ possesses the hereditary property.    
\end{proof}
We note that a similar argument can be applied, since $T_{\mathcal{C}}$ is a cotilting module, to show $T_{\mathcal{C}}$ possesses the co-hereditary property.  Next, we justify the naming of this property by proving a characterization of hereditary algebras.
\begin{theorem}
\label{Main3}
Let $\Lambda$ be an algebra.  Then $\Lambda$ is hereditary if and only if every tilting $\Lambda$-module possesses the hereditary property.
\end{theorem}
\begin{proof}
Assume every tilting $\Lambda$-module possesses the hereditary property.  Consider $\Lambda_{\Lambda}$ as a right $\Lambda$-module.  Clearly, $\Lambda_{\Lambda}$ is a tilting $\Lambda$-module and, for every $\Lambda$-module $M$, $M\in\Gen(\Lambda_{\Lambda})$.
 We have a short exact sequence $0\rightarrow (\Lambda_{\Lambda})' \rightarrow (\Lambda_{\Lambda})''\rightarrow M\rightarrow 0$ where $(\Lambda_{\Lambda})'$ and $(\Lambda_{\Lambda})''\in\add\Lambda_{\Lambda}$.  This implies $\pd_{\Lambda}M\leq1$.  Since $M$ was arbitrary, we conclude $\mathop{\text{gl.dim}}\Lambda\leq1$ and $\Lambda$ is a hereditary algebra.
\par
Assume $\Lambda$ is hereditary.  Let $T$ be a tilting $\Lambda$-module with $M\in\Gen T$.  Proposition $\ref{need}$ gives $0\rightarrow L \rightarrow T_0\rightarrow M\rightarrow 0$ with $T_0\in\add T$ and $L\in\Gen T$.  Applying $\Hom_{\Lambda}(-,\Gen T)$ gives the exact sequence $\Ext_{\Lambda}^1(T_0,\Gen T)\rightarrow\Ext_{\Lambda}^1(L,\Gen T)\rightarrow\Ext_{\Lambda}^2(M,\Gen T)$.  Proposition $\ref{Tilting}$ shows $\Ext_{\Lambda}^1(T_0,\Gen T)=0$ and $\Lambda$ being hereditary forces $\Ext_{\Lambda}^2(M, \Gen T)=0$.  This implies $\Ext_{\Lambda}^1(L,\Gen T)=0$ and another application of Proposition $\ref{Tilting}$ gives $L\in\add T$.  We conclude $T$ possesses the hereditary property.  Since $T$ was an arbitrary tilting $\Lambda$-module, we are done.
\end{proof}
The proof provides a sufficient condition for a titling module to possess the hereditary property.  Here, $\Lambda$ is an arbitrary algebra.
\begin{cor}
Let $\Lambda$ be an algebra and $T$ a tilting $\Lambda$-module.  Suppose $\id_{\Lambda}M\leq1$ for every $M\in\Gen T$.  Then $T$ possesess the hereditary property.
\end{cor}
\begin{proof}
Following the proof of Theorem $\ref{Main3}$, we have the following exact sequence $\Ext_{\Lambda}^1(T_0,\Gen T)\rightarrow\Ext_{\Lambda}^1(L,\Gen T)\rightarrow\Ext_{\Lambda}^2(M,\Gen T)$.  Once again, Proposition $\ref{Tilting}$ shows $\Ext_{\Lambda}^1(T_0,\Gen T)=0$ and our assumption forces $\Ext_{\Lambda}^2(M, \Gen T)=0$.  Proposition $\ref{Tilting}$ gives $L\in\add T$ and we conclude $T$ possesses the hereditary property.
\end{proof}
Given an algebra $\Lambda$,  a tilting $\Lambda$-module $T$ is $\it{separating}$ if the induced torsion pair, $(\mathcal{T}(T),\mathcal{F}(T))$, is splitting.
\begin{prop}
Let $\Lambda$ be an algebra.  Suppose there exists a separating tilting $\Lambda$-module $T$ that possesses the hereditary property.  Then $\mathop{\emph{gl.dim}}\Lambda\leq2$.
\end{prop}
\begin{proof}
Let $S$ be a simple $\Lambda$-module.  Suppose $S\in\Gen T$.  By assumption, there exists 
$0\rightarrow T'\rightarrow T''\rightarrow S\rightarrow 0$ in $\mathop{\text{mod}}\Lambda$ with $T'$ and $T''$ $\in \add T$.  Since $\pd_{\Lambda}T\leq1$, we must have $\pd_{\Lambda}S\leq2$.  Next, assume $S\in\Cogen(\tau_{\Lambda}T)$.  Since $T$ is separating, we know from Proposition $\ref{split}$ that $\tau_{\Lambda}S\in\Cogen(\tau_{\Lambda}T)$.  It is well know $\pd_{\Lambda}S\leq1$ if and only if $\Hom_{\Lambda}(D\Lambda,\tau_{\Lambda}S)=0$.  By the definition of a tilting module, $D\Lambda_{\Lambda}\in\Gen T$.  Since $D\Lambda_{\Lambda}\in\Gen T$ and $\tau_{\Lambda}S\in\Cogen(\tau_{\Lambda}T)$, we must have $\Hom_{\Lambda}(D\Lambda,\tau_{\Lambda}S)=0$.  Thus, $\pd_{\Lambda}S\leq1$.  By Theorem $\ref{Global}$, we conclude $\mathop{\text{gl.dim}}\Lambda\leq2$.

\end{proof}

\noindent Mathematics Faculty, University of Connecticut-Waterbury, Waterbury, CT 06702, USA
\it{E-mail address}: \bf{stephen.zito@uconn.edu}

\end{document}